\documentclass{amsart}
\usepackage{amsmath, amsthm, amssymb, mathrsfs, verbatim}
\usepackage{fullpage}
\usepackage{array}
\usepackage{hyperref}
\usepackage{tikz}
\usepackage{mathtools}
\usetikzlibrary{arrows,backgrounds}

\pgfdeclarelayer{bg}
\pgfsetlayers{bg,main}

\numberwithin{equation}{section}

\newcommand{\st}{~|~}

\newcommand{\R}{\mathbb{R}}

\newcommand{\sC}{\mathscr{C}}
\newcommand{\Z}{\mathbb{Z}}

\newcommand{\B}{\mathcal{B}}
\newcommand{\D}{\mathtt{Del}}

\renewcommand{\L}{\mathcal{L}}

\renewcommand{\phi}{\varphi}
\renewcommand{\epsilon}{\varepsilon}
\renewcommand{\emptyset}{\varnothing}

\newcommand{\dee}{\partial}

\newcommand{\Cone}{\mathtt{Cone}}
\DeclareMathOperator{\del}{\mathrm{del}}

\DeclareMathOperator{\In}{\mathrm{in}}

\DeclareMathOperator{\link}{\mathrm{link}}
\newcommand{\Link}{\mathtt{Link}}

\newcommand{\1}{^{-1}}

\newcommand{\la}{\langle}
\newcommand{\ra}{\rangle}

\newtheorem{thm}[equation]{Theorem}
\newtheorem{prop}[equation]{Proposition}
\newtheorem{lem}[equation]{Lemma}
\newtheorem{cor}[equation]{Corollary}

\theoremstyle{definition}

\newtheorem{defn}[equation]{Definition}
\newtheorem{ex}[equation]{Example}

\def\ZZ{\mathbb{Z}}
\newcommand{\garrow}[1]{\stackrel{#1}{\longrightarrow}}
\def\Q{\mathsf{Q}}
\tikzstyle{mutable}=[inner sep=0.5mm,circle,draw,minimum size=2mm]
\tikzstyle{frozen}=[inner sep=.9mm,rectangle,draw]
\theoremstyle{definition}
\newtheorem{rem}[equation]{Remark}
\newtheorem{rems}[equation]{Remarks}
\newcommand{\gmat}[2][ccccccccccccccccccccccccccccccccc]{\left(\begin{array}{#1} #2\\ \end{array}\right)}

\title[Lower bound cluster algebras]{Lower bound cluster algebras:\\presentations, cohen-macaulayness, and normality}
\author{Greg Muller}
\address[Greg Muller]{University of Michigan}
\email{morilac@umich.edu}

\author{Jenna Rajchgot}
\address[Jenna Rajchgot]{University of Michigan}
\email{rajchgot@umich.edu}

\author{Bradley Zykoski}
\address[Bradley Zykoski]{University of Virginia}
\email{bpz8nr@virginia.edu}
\thanks{$2010$ \emph{Mathematics Subject Classification.} Primary 13F60, Secondary 05E40,13F55}
\thanks{\emph{Keywords:} Cluster algebras, lower bound cluster algebras, combinatorial commutative algebra, Stanley-Reisner complexes}
\thanks{The third author was supported by NSF grant DMS-1001764, as an REU project at the University of Michigan in 2015.}

\begin{document}

\maketitle

\begin{abstract}
We give an explicit presentation for each lower bound cluster algebra. Using this presentation, we show that each  lower bound algebra Gr\"obner degenerates to the Stanley-Reisner scheme of a vertex-decomposable ball or sphere, and is thus Cohen-Macaulay. Finally, we use Stanley-Reisner combinatorics and a result of Knutson-Lam-Speyer to show that all lower bound algebras are normal.
\end{abstract}

\section{Introduction and statement of results}

Cluster algebras are a family of combinatorially-defined commutative algebras which were introduced by Fomin and Zelevinsky at the turn of the millennium to axiomatize and generalize patterns appearing in the study of dual canonical bases in Lie theory \cite{FZ02}.  Since their introduction, cluster algebras have been discovered in the rings of functions on many important spaces, such as semisimple Lie groups, Grassmannians, flag varieties, and Teichm\"uller spaces \cite{BFZ05,Sco06,GLS08,GSV05}.\footnote{A more interesting and morally correct statement is that each of these spaces possesses a stratification such that each stratum naturally has a cluster algebra in its ring of functions.}

In each of these examples, the cluster algebra is realized as the coordinate ring of a smooth variety.  This makes it all the more surprising that the varieties associated to general cluster algebras can be singular; in fact, they can possess such nightmarish pathologies as a non-Noetherian singularity \cite{MulLA}.  Various approaches have been introduced to mitigate this. 

\begin{itemize}
	\item Restricting to a subclass of cluster algebras with potentially better behavior: acyclic cluster algebras \cite{BFZ05}, locally acyclic cluster algebras \cite{MulLA,BMRS15}, or cluster algebras with a maximal green sequence \cite{BDP14,MulMGS}.
	\item Replacing the cluster algebra by a closely-related algebra with potentially better behavior: upper cluster algebras \cite{BFZ05,BMRS15}, the span of convergent theta functions \cite{GHKK}, or \textbf{lower bound algebras} \cite{BFZ05}.
\end{itemize}
In this note, we study the algebraic and geometric behavior of lower bound algebras.\footnote{More specifically, we consider lower bound algebras \emph{defined by a quiver} in the body of the paper, and consider the more general context of \emph{geometric type} in Appendix \ref{section: skew}.}

\subsection{Lower bound algebras}

Lower bound algebras were introduced in \cite{BFZ05} as a kind of `lazy approximation' of a cluster algebra, in the following sense.  A cluster algebra is defined to be the subalgebra of a field of rational functions generated by a (usually infinite) set of \emph{cluster variables}, produced by a recursive procedure called \emph{mutation}.  A lower bound algebra is defined to be the subalgebra generated by truncating this process at a specific finite set of steps.  The resulting algebra is contained in the associated cluster algebra and is manifestly finitely generated.

A lower bound algebra is constructed from an \textbf{ice quiver} $\Q$: this is a quiver (i.e. a finite directed graph) without loops or directed $2$-cycles, in which each vertex is designated \textbf{unfrozen} or \textbf{frozen}.
As a matter of convenience, we assume the vertices of $\Q$ have been indexed by the numbers $\{1,2,...,n\}$.  To each unfrozen vertex $i$, we associate a pair of monomials $p_i^+,p_i^-\in \ZZ[x_1,x_2,...,x_n]$ as follows.
\begin{equation}
p_i^+ \coloneqq \prod_{\stackrel{\text{arrows }a\in \Q}{\text{source}(a)=i}}x_{\text{target}(a)},\hspace{1cm} p_i^- \coloneqq \prod_{\stackrel{\text{arrows }a\in \Q}{\text{target}(a)=i}}x_{\text{source}(a)}
\end{equation}
Each vertex then determines a Laurent polynomial $x_i'$, called the \textbf{adjacent cluster variable at $i$}, which is defined by the following formula.\footnote{This is an abuse of terminology.  Technically speaking, a frozen vertex $i$ should not have an adjacent cluster variable $x_i'$, and instead we should include $x_i^{-1}$ as a generator (though this latter step is a matter of some debate).  We are streamlining the process by calling the inverse $x_i^{-1}$ `the adjacent cluster variable at $i$'.}
\begin{equation}\label{eq: mutation}
x_i' \coloneqq \left\{\begin{array}{cc}
x_i^{-1}(p_i^++p_i^-) & \text{if $i$ is unfrozen} \\
x_i^{-1} & \text{if $i$ is frozen}
\end{array}\right\}
\end{equation}
The \textbf{lower bound algebra} $\L(\Q)$ defined by $\Q$ is the subalgebra of $\ZZ[x_1^{\pm1},...,x_n^{\pm1}]$ generated by the variables $x_1,x_2,...,x_n$ and the adjacent cluster variables $x_1',x_2',...,x_n'$.

\begin{figure}[h!t]
\begin{tikzpicture}[scale=1.5]
	\node[mutable] (1)  at (-1,0) {$1$};
	\node[mutable] (2) at (0,0) {$2$};
	\node[mutable] (3) at (1,.5) {$3$};
	\node[frozen] (4) at (1,-.5) {$4$};
	\draw[-angle 90,relative,out=15,in=165] (1) to (2);
	\draw[-angle 90,relative,out=-15,in=-165] (1) to (2);
	\draw[-angle 90] (2) to (3);
	\draw[-angle 90] (2) to (4);
\end{tikzpicture}
\caption{An ice quiver (the unique frozen vertex is depicted as a square)}
\label{fig: mutationexample}
\end{figure}
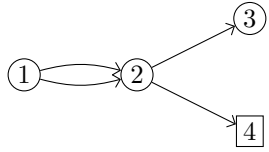

\begin{ex}
Consider the ice quiver $\Q$ in Figure \ref{fig: mutationexample}.  The four adjacent cluster variables are below.
\[ x_1' = \frac{x_2^2+1}{x_1},\;\;\; x_2'=\frac{x_3x_4+x_1^2}{x_2},\;\;\; x_3'=\frac{1+x_2}{x_3},\;\;\; x_4'=\frac{1}{x_4} \]
\end{ex}

%

\subsection{Relations in $\L(\Q)$}

We first consider the problem of finding relations among the generators of $\L(\Q)$.
%
%
Each adjacent cluster variable satisfies a \textbf{defining relation} immediately from its definition.
\begin{equation}
\forall i \text{ unfrozen} ,\; \;\; x_i'x_i = (p_i^++p_i^-) 
\end{equation}
\begin{equation}
\forall i \text{ frozen} ,\; \;\;x_i'x_i =1
\end{equation}
A more interesting class of relations is given by the following proposition.
\begin{prop}[The cycle relations]\label{prop: cyclerels}
For each directed cycle of unfrozen vertices $v_1\rightarrow v_2 \rightarrow \cdots \rightarrow v_k\rightarrow v_{k+1}=v_1$ in $\Q$, the following \textbf{cycle relation} holds.
\begin{equation}\label{eq: cyclerel}
\sum_{\stackrel{S\subset \{1,2,...,k\}}{S\cap (S+1)=\emptyset}} (-1)^{|S|}\left(\prod_{i\in S} \frac{p_{v_i}^+p_{v_{i+1}}^-}{x_{v_i}x_{v_{i+1}}}\right)\left(\prod_{i\not\in S\cup (S+1)} x'_{v_i}\right) = \prod_{i=1}^k\frac{p_{v_i}^+}{x_{v_i}}+ \prod_{i=1}^k\frac{p_{v_i}^-}{x_{v_i}}
\end{equation}
\end{prop}
\noindent We note that the expressions on either side reduce to polynomials in the generators, despite the presence of fractions.  Also note that choosing a different initial vertex $v_1$ in the same directed cycle does not change the corresponding cycle relation.


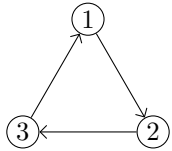
\begin{figure}[h!t]
\begin{tikzpicture}[scale=1]
	\node[mutable] (1)  at (90:1) {$1$};
	\node[mutable] (2) at (-30:1) {$2$};
	\node[mutable] (3) at (210:1) {$3$};
	\draw[-angle 90] (1) to (2);
	\draw[-angle 90] (2) to (3);
	\draw[-angle 90] (3) to (1);
\end{tikzpicture}
\caption{An ice quiver (no frozen vertices)}
\label{fig: 3cycle}
\end{figure}

\begin{rem}
A quiver $\Q$ is called \textbf{acyclic} if it has no directed cycles of unfrozen vertices.  The unifying theme of this paper is the use of the cycle relations to generalize results about $\L(\Q)$ which were previously known only when $\Q$ is acyclic (that is, when there are no cycle relations).
\end{rem}

\begin{ex}
Let $\Q$ be the ice quiver in Figure \ref{fig: 3cycle}.  The three adjacent cluster variables are 
\[ x_1'= \frac{x_2+x_3}{x_1},\;\;\; x_2'=\frac{x_3+x_1}{x_2},\;\;\; x_3'=\frac{x_1+x_2}{x_3} \]
The defining relations here are obtained by clearing the denominators above.  A non-trivial directed $3$-cycle starting at any vertex determines the cycle relation
\[ x_1'x_2'x_3' - x_1'-x_2'-x_3'=2\]
which may be verified by direct computation.
\end{ex}

\subsection{A presentation of $\L(\Q)$}

We may ask whether there are other relations in $\L(\Q)$ that are not an immediate consequence of the preceding relations; or more concretely, whether the defining relations and the cycle relations generate the entire ideal of relations among the generators.

Explicitly, we consider the homomorphism of rings\footnote{The $y$-variables introduced here have no relation to the \emph{$y$-variables} or \emph{coefficient variables} introduced in \cite{FZ07}.}
\[ \pi :\ZZ[x_1,x_2,...,x_n,y_1,y_2,...y_n]\longrightarrow \ZZ[x_1^{\pm1},x_2^{\pm1},...,x_n^{\pm1} ]\]
\[ \forall i, \;\;\; \pi(x_i) = x_i,\;\;\; \pi(y_i)=x_i'\]
The image of this homomorphism is $\L(\Q)$, and so $K_\Q \coloneqq \ker(\pi)$ is the \textbf{ideal of relations} among the generators of $\L(\Q)$, where each adjacent cluster variable $x_i'$ has been replaced by the abstract variable $y_i$.  The homomorphism $\pi$ descends to an isomorphism
\[ \ZZ[x_1,x_2,...,x_n,y_1,y_2,...,y_n]/K_\Q\garrow{\sim} \L(\Q) \]

We will say a directed cycle $v_1\rightarrow v_2 \rightarrow \cdots \rightarrow v_k\rightarrow v_{k+1}=v_1$ is \textbf{vertex-minimal} if no vertex appears more than once and there is no directed cycle whose vertex set is a proper subset of $\{v_1,v_2...,v_k\}$.

\begin{thm}\label{thm: relations}
The ideal of relations $K_\Q$ is generated by the following elements.
\begin{itemize}
	\item For each unfrozen vertex $i$, 
\begin{equation}\label{eq: mutationrel}
y_ix_i - p_i^+-p_i^-
\end{equation}
	\item For each frozen vertex $i$, 
\begin{equation}\label{eq: inverserel}
y_ix_i-1 
\end{equation}
	\item For each vertex-minimal 
	directed cycle of unfrozen vertices $v_1\rightarrow v_2 \rightarrow \cdots \rightarrow v_k\rightarrow v_{k+1}=v_1$,
\begin{equation}\label{eq: cyclerel}
\sum_{\stackrel{S\subset \{1,2,...,k\}}{S\cap (S+1)=\emptyset}} (-1)^{|S|}\left(\prod_{i\in S} \frac{p_{v_i}^+p_{v_{i+1}}^-}{x_{v_i}x_{v_{i+1}}}\right)\left(\prod_{i\not\in S\cup (S+1)} y_{v_i}\right) - \prod_{i=1}^k\frac{p_{v_i}^+}{x_{v_i}}- \prod_{i=1}^k\frac{p_{v_i}^-}{x_{v_i}}
\end{equation}
which simplifies to a polynomial.
\end{itemize}
\end{thm}
\noindent The theorem is true without the vertex-minimal condition, which is used here to reduce the set of relations.

\begin{ex}\label{ex:3vertexgrobner}
Let $\Q$ be the ice quiver in Figure \ref{fig: 3cycle}.  By Theorem \ref{thm: relations}, $\L(\Q)$ is isomorphic to the quotient of $\ZZ[x_1,x_2,x_3,y_1,y_2,y_3]$ by the ideal $K_\Q$ generated by the following $4$ relations.
\[ K_\Q =\langle y_1x_1-x_2-x_3,y_2 x_2-x_3-x_1, y_3x_3-x_1-x_2, y_1y_2y_3-y_1-y_2-y_3-2\rangle \]
\end{ex}

\subsection{A Gr\"obner basis for $K_\Q$}

We prove Theorem \ref{thm: relations} by means of a stronger result, that the given generators are a Gr\"obner basis for the ideal of relations $K_\Q$. 
Recall that, given a polynomial ring with a monomial order $<$, a \textbf{Gr\"obner basis} of an ideal $I$ is a generating set $\{g_1,g_2,..,g_k\}$ of $I$ satisfying the additional condition that $\{\text{in}_<(g_1),\text{in}_<(g_2),...,\text{in}_<(g_k)\}$ is a generating set of $\text{in}_<(I)$.  

The monomial orders relevant to us are those in which the $y$-variables are much more expensive than the $x$-variables, that is $\mathbf{x}^{\mathbf{\alpha}}\mathbf{y}^{\mathbf{\beta}}> \mathbf{x}^{\mathbf{\gamma}}\mathbf{y}^{\mathbf{\delta}}$ whenever $\sum_i \beta_i$ is larger than $\sum_i \delta_i$.
An example of such a monomial order is the lexicographical order where the variables are ordered by
\[
y_1>y_2>\cdots>y_n>x_1>x_2>\cdots>x_n.
\]


\begin{thm}\label{thm: grobner}
For any monomial order of $\ZZ[x_1,x_2,...,x_n,y_1,y_2,...,y_n]$ in which all of the $y$-variables are much more expensive than all of the $x$-variables, the polynomials given in Theorem \ref{thm: relations} are a Gr\"obner basis for $K_\Q$. Consequently, the initial ideal $\text{in}_< K_\Q$ is squarefree monomial ideal with generating set
\[
\{x_iy_i \mid 1\leq i\leq n\}\cup \{ y_{v_1}y_{v_2}\cdots y_{v_k}\mid v_1\rightarrow v_2\rightarrow \cdots \rightarrow v_k\rightarrow v_{k+1} = v_1  \textrm{ is a vertex-minimal cycle in }\Q\}.
\]
\end{thm}

\begin{rem}
When $\Q$ is acyclic, Theorem \ref{thm: grobner} specializes to Corollary 1.17 in \cite{BFZ05}.  The proof of Theorem \ref{thm: grobner} given in Section \ref{sect:generators} uses \cite[Corollary 1.17]{BFZ05} in an essential way, so our proof is not independent of the original result.
%
\end{rem}

\subsection{Simplicial complexes and Cohen-Macaulayness of lower bound algebras}
From here on, we work over a field $\mathbb{K}$, and consider the $\mathbb{K}$-algebra $\mathcal{L}(\Q)\otimes_\mathbb{Z}\mathbb{K}$. We will still refer to this algebra as a lower bound algebra and, though a slight abuse of notation, will simply denote it by $\mathcal{L}(\Q)$. Similarly, we let $K_\Q$ denote the associated lower bound ideal, so that it is the kernel of the map $\pi:\mathbb{K}[x_1,\dots,x_n,y_1\dots,y_n]\rightarrow \mathcal{L}(\Q)$.

To a squarefree monomial ideal $I\subseteq \mathbb{K}[z_1,\dots,z_n]$, one can associate a simplicial complex $\Delta$ on the vertex set $\{z_1,\dots,z_n\}$. This simplicial complex is called the \textbf{Stanley-Reisner complex} and is defined as follows: $\{z_{i_1},\dots z_{i_r}\}$ is a face of $\Delta$ if and only if the monomial $z_{i_1}\cdots z_{i_r}\notin I$. Observe that the minimal non-faces of $\Delta$ are in one-to-one correspondence with a minimal generating set of $I$. For further information on Stanley-Reisner complexes, see the textbook \cite[Chapter 1]{MillerSturmfels}.

Whenever the Stanley-Reisner complex is a simplicial ball or a sphere,\footnote{More precisely, we mean the geometric realization of the simplicial complex is homeomorphic to a ball or sphere, respectively.  Whenever we refer to a simplicial complex as a topological object, we more precisely mean its geometric realization.} the corresponding face ring $\mathbb{K}[z_1,\dots, z_n]/I$ is a Cohen-Macaulay ring \cite{Munkres}. Furthermore, when $I = \textrm{in}_< J$ for some ideal $J\subseteq \mathbb{K}[z_1,\dots, z_n]$, we may also conclude that $J$ itself is Cohen-Macaulay (see eg. \cite[Proposition 3.1]{BC03}).

\begin{ex}
We continue Example \ref{ex:3vertexgrobner} and observe that the initial ideal $\textrm{in}_< K_\Q$ (for a term order $<$ as in Theorem \ref{thm: grobner}) is $\langle x_1y_1, x_2y_2, x_3y_3, y_1y_2y_3\rangle$. The facets (i.e. the maximal faces) of the associated Stanley-Reisner complex are precisely those $\{z_1,z_2,z_3\}$ where $z_i$ is either $x_i$ or $y_i$ and at least one of $z_1, z_2, z_3$ is an $x_i$. This Stanley-Reisner complex is readily seen to be a simplicial ball, and is pictured in Figure \ref{fig:3vertexball}.

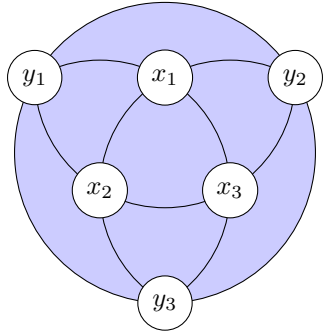
\begin{figure}
\begin{tikzpicture}
	\begin{scope}
		\draw[fill=blue!20] (0,0) circle (2);
		\clip (0,0) circle (2);
		\draw (90:1) circle (1.732);
		\draw (-30:1) circle (1.732);
		\draw (210:1) circle (1.732);
	\end{scope}
	\node[draw,circle,fill=white] at (90:1) {$x_1$};
	\node[draw,circle,fill=white] at (210:1) {$x_2$};
	\node[draw,circle,fill=white] at (330:1) {$x_3$};
	\node[draw,circle,fill=white] at (150:2) {$y_1$};
	\node[draw,circle,fill=white] at (30:2) {$y_2$};
	\node[draw,circle,fill=white] at (270:2) {$y_3$};

\end{tikzpicture}

\caption{The Stanley-Reisner complex for $\la x_1y_1, x_2y_2, x_3y_3, y_1y_2y_3 \ra$}
\label{fig:3vertexball}
\end{figure}
\end{ex}

This example generalizes, and we are able to conclude that all lower bound algebras are Cohen-Macaulay. 

\begin{thm}\label{thm:CM}
Let $\Q$ be a quiver with $n$ vertices, let $K_\Q\subseteq \mathbb{K}[x_1,\dots,x_n,y_1,\dots,y_n]$ be its ideal of relations, and let $<$ be any monomial order where the $y$-variables are much more expensive than the $x$-variables. Let $\Delta_\Q$ be the Stanley-Reisner complex of the squarefree monomial ideal $\textrm{in}_< K_\Q$.
\begin{enumerate}
\item If $\Q$ is acyclic, then $\Delta_\Q$ is a simplicial sphere.
\item If $\Q$ is not acyclic, then $\Delta_\Q$ is a simplicial ball.
\end{enumerate}
\end{thm}

\noindent We show additional properties of $\Delta_\Q$.  If $\Q$ is acyclic, then $\Delta_\Q$ is the boundary of a cross-polytope.  In both cases, $\Delta_\Q$ satisfies the stronger condition of \emph{vertex-decomposibility}. Details are in Section \ref{sect:combinatorics}.

\begin{cor}
For any $\Q$ the lower bound algebra $\L(\Q)$ over a field $\mathbb{K}$ is Cohen-Macaulay.
\end{cor}



\begin{rems}
\begin{enumerate}
\item We prove Theorem \ref{thm:CM} by way of a more general result, which gives a larger class of simplicial complexes which are automatically vertex-decomposable balls (see Theorem \ref{homeo}).
\item When $\Q$ is acyclic, $\L(\Q)$ was already known to be Cohen-Macaulay; specifically, \cite[Corollary 1.17]{BFZ05} implies that $\L(\Q)$ is a complete intersection, and, consequently, that it is Cohen-Macaulay.
\end{enumerate}
\end{rems}

\subsection{Normality of lower bound algebras}

Our last main result is the normality of the $\mathbb{K}$-algebra $\L(\Q)$.

\begin{thm}\label{thm:normalityOfLowerBounds}
Every lower bound cluster algebra defined by a quiver is normal.
\end{thm}

Since $\L(\Q)= \mathbb{K}[x_1,...,x_n,y_1,...,y_n]/K_\Q$ is finitely-generated, Serre's Criterion reduces normality to a pair of geometric conditions on the variety $\mathbb{V}(K_\Q)$.
\begin{itemize}
	\item (R1) The variety $\mathbb{V}(K_\Q)$ has no codimension\--$1$ singularities.
	\item (S2) Any regular function on an open subset in $\mathbb{V}(K_\Q)$ with codimension\--$2$ complement extends to a regular function on all of $\mathbb{V}(K_\Q)$.
\end{itemize}
The Cohen-Macaulay property implies the S2 condition, and so normality of $\L(\Q)$ reduces to proving there are no codimension\--$1$ singularities.  

As with Cohen-Macaulayness, this geometric question will be reduced to Stanley-Reisner combinatorics, along  with a result of Knutson-Lam-Speyer \cite[Proposition 8.1]{KLS-Richardson}. See Section \ref{sect:normality} for further information.

%
%

\begin{rem}
Like our other results, Theorem \ref{thm:normalityOfLowerBounds} is only new in the case of non-acyclic $\Q$. In the acyclic setting, $\L(\Q)$ is equal to its upper cluster algebra\footnote{This was proven with an additional assumption in \cite[Thm. 1.18]{BFZ05}, and without said assumption in \cite{MulAU}.},  which is a normal domain by \cite[Prop. 2.1]{MulLA}.
\end{rem}

\subsection*{Structure of paper}

Section \ref{section: algebra} considers relations in $\L(\Q)$ and proves the associated results: Proposition \ref{prop: cyclerels}, Theorem \ref{thm: relations} and Theorem \ref{thm: grobner}.  Section \ref{sect:combinatorics} introduces the relevant combinatorial tools, leading to the proof of Theorem \ref{thm:CM}.  Section \ref{sect:normality} addresses normality, proving Theorem \ref{thm:normalityOfLowerBounds}.  

The paper concludes with a pair of appendices which frame the scope of the paper.  Appendix \ref{section: singularity} considers the singularities of lower bound algebras directly, and provides an example to suggest this is a difficult problem.  Appendix \ref{section: skew} explains how the results of the paper can be extended to \emph{skew-symmetrizable} lower bound algebras, which are more general but also somewhat less intuitive.

\subsection*{Acknowledgements}
This paper is the result of a summer 2015 REU project at the University of Michigan. This project was supported by Karen Smith's NSF grant DMS-1001764. We are also grateful to Sergey Fomin for numerous helpful comments on an early draft of this note.

\section{Presentations and Gr\"obner Bases}\label{section: algebra}
\subsection{Choice graphs and cycle relations}
\noindent
To every cycle of $\Q$ there exists a corresponding relation in $\L(\Q)$. Let $\Q$ have a directed cycle $v_1 \to v_2 \to \cdots \to v_{k-1} \to v_k \to v_1$. By giving an alternate presentation for the product
\begin{equation}\label{Product}
\prod_{i=1}^k x_{v_i}' = \prod_{i=1}^k x_{v_i}\1 (p_{v_i}^+ + p_{v_i}^-),
\end{equation}
we acquire a nontrivial relation that holds in $\L(\Q)$. It is our goal to expand the right-hand product as a sum, and from this, Proposition \ref{prop: cyclerels} will follow.
Each term of the expansion of this product represents a choice, for each $i$, of either $p_{v_i}^+$ or $p_{v_i}^-$ from $x_{v_i}\1(p_{v_i}^+ + p_{v_i}^-)$. Therefore, to each term, we associate a directed graph with $\Z/k\Z$ as its vertex set and $\{(i,i\pm1)\}_{i=1}^n$ as its set of arrows, where the sign of $\pm$ corresponds to the abovementioned choice of $p_{v_i}^+$ (corresponding to the positive sign because $x_{v_{i+1}}$ divides $p_{v_i}^+$) or $p_{v_i}^-$ (negative sign, since $x_{v_{i-1}}$ divides $p_{v_i}^-$). Call these graphs the \textbf{choice graphs} of the terms of the expansion of \eqref{Product}, and let $\mathfrak C$ denote the set of all choice graphs. Formally, we write the correspondence between choice graphs and terms as a function
\begin{equation}
M: \mathfrak C \to \{\text{monomials in }\Z[x_1^{\pm1},\dots,x_k^{\pm1}]\},\quad M(\mathsf g) = \prod_{i=1}^k x_{v_i}\1p_{v_i}^{\mathrm{sign}_{\mathsf g}(i)},
\end{equation}
where
\begin{equation}
\mathrm{sign}_{\mathsf g}(i) = \begin{cases} + & \text{if } (i,i+1) \in \mathsf g,\\ - & \text{if } (i,i-1) \in \mathsf g. \end{cases}.
\end{equation}
\begin{ex}
Let $\Q$ be the quiver on $\Z/6\Z$ whose set of arrows is $\{(j,j+1)\}_{j=1}^6$. The choice graph in Figure \ref{fig: choice graph} represents the term
\[
(x_{1}\1 p_{1}^+)(x_{2}\1 p_{2}^+)(x_{3}\1 p_{3}^-)(x_{4}\1 p_{4}^-)(x_{5}\1 p_{5}^-)(x_{6}\1 p_{6}^-) = x_{1}\1 p_{1}^+ \frac{p_{2}^+ p_{3}^-}{x_3 x_2} p_4^+ \frac{p_5^-}{x_4} \frac{p_6^-}{x_5} x_{6}\1.
\]
\end{ex}

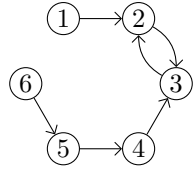
\begin{figure}
\begin{tikzpicture}[scale=1]
	\node[mutable] (1) at (120:1) {$1$};
	\node[mutable] (2) at (60:1) {$2$};
	\node[mutable] (3) at (0:1) {$3$};
	\node[mutable] (4) at (300:1) {$4$};
	\node[mutable] (5) at (240:1) {$5$};
	\node[mutable] (6) at (180:1) {$6$};
	\draw[-angle 90] (2) to [bend left] (3);
	\draw[-angle 90] (1) to (2);
	\draw[-angle 90] (5) to (4);
	\draw[-angle 90] (4) to (3);
	\draw[-angle 90] (6) to (5);
	\draw[-angle 90] (3) to [bend left] (2);
\end{tikzpicture}
\caption{The choice graph for $(x_{1}\1 p_{1}^+)(x_{2}\1 p_{2}^+)(x_{3}\1 p_{3}^-)(x_{4}\1 p_{4}^-)(x_{5}\1 p_{5}^-)(x_{6}\1 p_{6}^-)$}
\label{fig: choice graph}
\end{figure}
By our construction of $\mathfrak C$ and $M$, we have that \eqref{Product} may be written as
\begin{equation}\label{eq: choice sum}
\prod_{i=1}^k x_{v_i}\1 (p_{v_i}^+ + p_{v_i}^-) = \sum_{\mathsf g \in \mathfrak C} M(\mathsf g),
\end{equation}
and so it is our goal to expand the sum on the right of \eqref{eq: choice sum}. The following proposition gives us such an expansion.
\begin{prop}\label{cycle}
We may expand \eqref{eq: choice sum} as follows:
\begin{equation}\label{cyclerel-eq}
\prod_{i=1}^k x_{v_i}\1 (p_{v_i}^+ + p_{v_i}^-) 
= \prod_{i=1}^k \frac{p_{v_i}^+}{x_{v_{i-1}}} + \prod_{i=1}^k \frac{p_{v_i}^-}{x_{v_{i+1}}} + \sum_{\stackrel{\emptyset \neq S\subset \{1,2,...,k\}}{S\cap (S+1)=\emptyset}}(-1)^{|S|+1} \left(\prod_{i\in S} \frac{p_{v_i}^+}{x_{v_{i+1}}}\frac{p_{v_{i+1}}^-}{x_{v_i}}\right) \left(\prod_{i \notin S \cup (S+1)} x_{v_i}'\right)
\end{equation}
\end{prop}
\begin{proof}
\indent Note that a choice graph has a directed 2-cycle if and only if its associated term has a factor of the form $(x_{v_i}\1 p_{v_i}^+)(x_{v_{i+1}}\1 p_{v_{i+1}}^-)$, which is a monomial $\frac{p_{v_i}^+}{x_{v_{i+1}}} \frac{p_{v_{i+1}}^-}{x_{v_i}}$ in the variables $x_{v_1},\dots,x_{v_n}$ since $x_{v_{i+1}}$ divides $p_{v_i}^+$ and $x_{v_i}$ divides $p_{v_{i+1}}^-$.
Also notice that no two 2-cycles may share vertices, since each vertex meets the tail of one and only one arrow. It follows that a choice graph may have at most $\lfloor \frac k2 \rfloor$ 2-cycles. Finally notice that a 2-cycle may only exist on adjacent vertices. Now, let $\mathfrak S_j$ denote the collection of subsets $S$ of $\{1,\dots,k\}$ of size $j\geq1$ such that $S \cap (S+1) = \emptyset$, and let $\mathfrak S = \bigcup_{j=1}^{\lfloor \frac k2 \rfloor} \mathfrak S_j$. These subsets $S$ correspond to the ``left endpoints'' of the 2-cycles in choice graphs with $j$ 2-cycles.\\
\indent Let $C(S)$ be the set of all choice graphs with a 2-cycle on every pair $\{i,i+1\}$ for $i \in S \in \mathfrak S$, and let $C_0$ be the set of choice graphs with no 2-cycles, then we have
\begin{equation}\label{eq: separation}
\mathfrak C = C_0 \cup \left(\bigcup_{S \in \mathfrak S_1}C(S)\right) \cup \left(\bigcup_{S \in \mathfrak S_2}C(S)\right) \cup \cdots \cup \left(\bigcup_{S \in \mathfrak S_{\lfloor \frac n2 \rfloor}} C(S)\right).
\end{equation}
Since every $\mathsf g \in C(S)$ has a pair of arrows $(i,i+1), (i+1,i)$ for every $i \in S$, we have
\begin{equation}\label{eq: partial choice}
\sum_{\mathsf g \in C(S)} M(\mathsf g) = \left(\prod_{i\in S} \frac{p_{v_i}^+}{x_{v_{i+1}}}\frac{p_{v_{i+1}}^-}{x_{v_i}}\right) \left(\prod_{i \notin S \cup (S+1)} x_{v_i}\1\left(p_{v_i}^+ + p_{v_i}^-\right)\right).
\end{equation}
We also have 
\[
\sum_{\mathsf g \in C_0} M(\mathsf g) = \prod_{i=1}^k \frac{p_{v_i}^+}{x_{v_{i-1}}} + \prod_{i=1}^k \frac{p_{v_i}^-}{x_{v_{i+1}}},
\]
since there are precisely two choice graphs with no 2-cycles, corresponding to a consistent choice of either $+$ or $-$.
We wish use \eqref{eq: separation} to write \eqref{eq: choice sum} as a sum with summands of the form \eqref{eq: partial choice}. Such a sum must have precisely one term corresponding to each member of $\mathfrak C$. However, there is a certain amount of overcounting involved in \eqref{eq: separation}, since for any $S \in \mathfrak S$, we have
\begin{equation}\label{overcounting}
C(S) \subseteq C(S \smallsetminus \{i\}) \quad \text{for every } i \in S.
\end{equation}
We therefore proceed iteratively by way of the inclusion-exclusion principle. We first include a summand corresponding to $C(S)$ for every $S \in \mathfrak S_1$:
\[
T_1 = \sum_{S \in \mathfrak S_1} \sum_{\mathsf g \in C(S)} M(\mathsf g).
\]
As (\ref{overcounting}) shows, for every $S \in \mathfrak S_2$, the summand $T_1$ contains two terms corresponding to each element of $C(S)$. Therefore, we now exclude a summand for each $C(S)$, $S \in \mathfrak S_2$:
\[
T_2 = T_1 - \sum_{S \in \mathfrak S_2} \sum_{\mathsf g \in C(S)} M(\mathsf g).
\]
Again, (\ref{overcounting}) shows us that, for every $S \in \mathfrak S_3$, the summand $T_2$ excludes one term too many for each element of $C(S)$. We therefore define $T_3$ accordingly, and continue the process of inclusion and exclusion until we obtain
\[
T_{\lfloor \frac k2 \rfloor} = \sum_{j=1}^{\lfloor \frac k2 \rfloor}\left((-1)^{j+1} \sum_{S \in \mathfrak S_j} \sum_{\mathsf g \in C(S)} M(\mathsf g)\right).
\]
Finally, we must include a term of $\sum_{\mathsf g \in C_0} M(\mathsf g)$ corresponding to $C_0$, and we conclude that
\begin{align*}
\sum_{\mathsf g \in \mathfrak C} M(\mathsf g) &= \sum_{\mathsf g \in C_0} M(\mathsf g) + T_{\lfloor \frac k2 \rfloor} \\
\prod_{i=1}^k x_{v_i}\1 (p_{v_i}^+ + p_{v_i}^-) &= \prod_{i=1}^k \frac{p_{v_i}^+}{x_{v_{i-1}}} + \prod_{i=1}^k \frac{p_{v_i}^-}{x_{v_{i+1}}} + \sum_{j=1}^{\lfloor \frac k2 \rfloor}\left((-1)^{j+1} \sum_{S \in \mathfrak S_j} \left(\prod_{i\in S} \frac{p_{v_i}^+}{x_{v_{i+1}}}\frac{p_{v_{i+1}}^-}{x_{v_i}}\right) \left(\prod_{i \notin S \cup (S+1)} \!\!\!x_{v_i}'\right)\right) \\
&= \prod_{i=1}^k \frac{p_{v_i}^+}{x_{v_{i-1}}} + \prod_{i=1}^k \frac{p_{v_i}^-}{x_{v_{i+1}}} + \sum_{\stackrel{\emptyset \neq S\subset \{1,2,...,k\}}{S\cap (S+1)=\emptyset}} \left((-1)^{|S|+1} \left(\prod_{i\in S} \frac{p_{v_i}^+}{x_{v_{i+1}}}\frac{p_{v_{i+1}}^-}{x_{v_i}}\right) \left(\prod_{i \notin S \cup (S+1)} \!\!\!x_{v_i}'\right)\right).
\end{align*}
\end{proof}
Proposition \ref{prop: cyclerels} follows, since we now have
\begin{align*}
\prod_{i=1}^k x_{v_i}\1 (p_{v_i}^+ + p_{v_i}^-) - \sum_{\stackrel{\emptyset \neq S\subset \{1,2,...,k\}}{S\cap (S+1)=\emptyset}} \left((-1)^{|S|+1} \left(\prod_{i\in S} \frac{p_{v_i}^+}{x_{v_{i+1}}}\frac{p_{v_{i+1}}^-}{x_{v_i}}\right) \left(\prod_{i \notin S \cup (S+1)} \!\!\!x_{v_i}'\right)\right)&= \prod_{i=1}^k \frac{p_{v_i}^+}{x_{v_{i-1}}} + \prod_{i=1}^k \frac{p_{v_i}^-}{x_{v_{i+1}}} \\
\sum_{\stackrel{S\subset \{1,2,...,k\}}{S\cap (S+1)=\emptyset}} \left((-1)^{|S|} \left(\prod_{i\in S} \frac{p_{v_i}^+}{x_{v_{i+1}}}\frac{p_{v_{i+1}}^-}{x_{v_i}}\right) \left(\prod_{i \notin S \cup (S+1)} \!\!\!x_{v_i}'\right)\right) &= \prod_{i=1}^k\frac{p_{v_i}^+}{x_{v_{i-1}}}+ \prod_{i=1}^k\frac{p_{v_i}^-}{x_{v_{i+1}}}.
\end{align*}


\subsection{Generators for $K_\Q$}\label{sect:generators}
Now, in addition to the \textbf{defining polynomials} $y_ix_i - (p_i^+ + p_i^-)$, $y_ix_i - 1$ given by the defining relations (\ref{eq: mutationrel}) and (\ref{eq: inverserel}), we have by Proposition \ref{cycle} that the ideal of relations $K_\Q$ also contains the \textbf{cycle polynomials}. We define the cycle polynomials in $\Z[x_1,x_2,\dots,x_n,y_1,y_2,\dots,y_n]$ to be those polynomials coming from vertex-minimal directed cycles whose images under $\pi$ vanish by virtue of (\ref{cyclerel-eq}). That is, for every vertex-minimal  directed cycle of unfrozen vertices $v_1 \to v_2 \to \cdots \to v_k \to v_1$ in $\Q$, we have the cycle polynomial
\[
\sum_{\stackrel{S\subset \{1,2,...,k\}}{S\cap (S+1)=\emptyset}} (-1)^{|S|}\left(\prod_{i\in S} \frac{p_{v_i}^+}{x_{v_{i+1}}}\frac{p_{v_{i+1}}^-}{x_{v_i}}\right)\left(\prod_{i\not\in S\cup (S+1)} y_{v_i}\right) - \prod_{i=1}^k\frac{p_{v_i}^+}{x_{v_{i-1}}}- \prod_{i=1}^k\frac{p_{v_i}^-}{x_{v_{i+1}}}.
\]
Note again that this expression reduces to a polynomial in the $x$- and $y$-variables because each $x_{v_{i+1}}$ divides $p_{v_i}^+$ and each $x_{v_i}$ divides $p_{v_{i+1}}^-$. In Table 1, we present the cycle polynomials given by some basic ice quivers.
\begin{table}
\begin{center}
\caption{Cycle polynomials in several examples}
\begin{tabular}{| >{\centering\arraybackslash}m{2in} | >{\centering\arraybackslash}m{2in} |}
\hline\vspace{.1cm}
\begin{tikzpicture}[scale=1]
	\node[mutable] (1) at (90:1) {$1$};
	\node[mutable] (2) at (330:1) {$2$};
	\node[mutable] (3) at (210:1) {$3$};
	\draw[-angle 90] (1) to (2);
	\draw[-angle 90] (2) to (3);
	\draw[-angle 90] (3) to (1);
\end{tikzpicture} & $y_1y_2y_3 - y_1 - y_2 - y_3 - 2$ \\
\hline\vspace{.1cm}
\begin{tikzpicture}[scale=1]
	\node[mutable] (1) at (135:1) {$1$};
	\node[mutable] (2) at (45:1) {$2$};
	\node[mutable] (3) at (315:1) {$3$};
	\node[mutable] (4) at (225:1) {$4$};
	\draw[-angle 90] (1) to (2);
	\draw[-angle 90] (2) to (3);
	\draw[-angle 90] (3) to (4);
	\draw[-angle 90] (4) to (1);
\end{tikzpicture} & $y_1y_2y_3y_4 - y_1y_2 - y_1y_4 - y_2y_3 - y_3y_4$ \\
\hline\vspace{.1cm}
\begin{tikzpicture}[scale=1]
	\node[mutable] (1) at (90:1) {$1$};
	\node[mutable] (2) at (18:1) {$2$};
	\node[mutable] (3) at (306:1) {$3$};
	\node[mutable] (4) at (234:1) {$4$};
	\node[mutable] (5) at (162:1) {$5$};
	\draw[-angle 90] (1) to (2);
	\draw[-angle 90] (2) to (3);
	\draw[-angle 90] (3) to (4);
	\draw[-angle 90] (4) to (5);
	\draw[-angle 90] (5) to (1);
\end{tikzpicture} & $y_1y_2y_3y_4y_5 - y_1y_2y_3 - y_1y_2y_5 - y_2y_3y_4 - y_3y_4y_5 + y_1 + y_2 + y_3 + y_4 + y_5 - 2$ \\
\hline\vspace{.1cm}
\begin{tikzpicture}[scale=1]
	\node[mutable] (2) {$2$};
	\node[mutable] (1) [above left of=2] {$1$};
	\node[mutable] (4) [above right of=2] {$4$};
	\node[mutable] (3) [below right of=2] {$3$};
	\node[mutable] (5) [below left of=2] {$5$};
	\draw[-angle 90] (1) to (2);
	\draw[-angle 90] (2) to (3);
	\draw[-angle 90] (3) to (4);
	\draw[-angle 90] (4) to (2);
	\draw[-angle 90] (2) to (5);
	\draw[-angle 90] (5) to (1);
\end{tikzpicture} & $y_1y_2y_5 - y_1x_3 - y_2 - y_5x_4 - x_3 - x_4$ and $y_2y_3y_4 - y_2 - y_3x_1 - y_4x_5 - x_1 - x_5$ \\
\hline
\end{tabular}
\end{center}
\label{table: polynomials}
\end{table}
\\ 
\indent We now obtain a presentation for the ideal of relations $K_\Q$ and for the initial ideal $\In_< K_\Q$, where $<$ is a monomial order in which the $y$-variables are much more expensive than the $x$-variables.
This presentation will suffice to prove Theorems \ref{thm: relations} and \ref{thm: grobner}. We first require the following standard lemma.
\begin{lem}\label{basislemma}
Let $J$ and $L$ be ideals in a polynomial ring. Suppose that $J \subseteq L$ and $\In_< J  = \In_< L$. Then $J = L$.
\end{lem}
\begin{proof}
Let $G$ be a Gr\"obner basis for $J$ and let $f \in L$. Since $\In_< G$ generates $\In_< L$, dividing $f$ by $G$ gives a remainder of $0$, and so $f \in J$.
\end{proof}
\begin{thm}\label{basisthm}
Given an ice quiver $\Q$ on $n$ vertices, the defining polynomials together with the cycle polynomials form a Gr\"obner basis for $K_\Q = \ker \pi$, where
\[ \pi :\ZZ[x_1,x_2,...,x_n,y_1,y_2,...y_n]\longrightarrow \ZZ[x_1^{\pm1},x_2^{\pm1},...,x_n^{\pm1} ]\]
\[ \forall i, \;\;\; \pi(x_i) = x_i,\;\;\; \pi(y_i)=x_i'.\]
\end{thm}
\begin{proof}
Let $J$ be the ideal of $\Z[x_1,x_2,\dots,x_n,y_1,y_2,\dots,y_n]$ that is generated by the set $G$ of defining and cycle polynomials. We know that $J \subseteq K_\Q$, and therefore that $\In_< J \subseteq \In_< K_\Q$. Let $M$ be the monomial ideal generated by the initial terms of the polynomials in $G$.  We know that $M \subseteq \In_< J \subseteq \In_< K_\Q$, so we would like to show that $\In_< K_\Q \subseteq M$. 

Assume (for the purpose of contradiction) that there is some $f\in K_\Q$ such that $\In_<(f)\not\in M$.  We may write
(assume all $a$, $b$ and $\lambda$ non-zero for simplicity)
\begin{equation}\label{form}
\In_<(f)= \lambda x_{i_1}^{a_{i_1}} x_{i_2}^{a_{i_2}} \cdots x_{i_k}^{a_{i_k}} y_{j_1}^{b_{j_1}} y_{j_2}^{b_{j_2}} \cdots y_{j_\ell}^{b_{j_\ell}}
\end{equation}
Note that $\{j_1,j_2,...,j_\ell\}$ cannot contain the indices of a directed cycle of unfrozen vertices. Otherwise, it would also contain the indices of a vertex-minimal directed cycle, and $\In_<(f)$ would be a multiple of the initial term of a cycle polynomial, contradicting the assumption that $\In_<(f)\not\in M$.
%

Let $Y\subset [n]$ be the indices of unfrozen vertices which are not in $\{j_1,j_2,...,j_\ell\}$, and let $\Q'$ be the ice quiver obtained by freezing the vertices in $\Q$ indexed by $Y$.  By the preceding observation, $\Q'$ is an acyclic quiver.
%
%
There is a natural inclusion
\[ \L(\Q)\hookrightarrow \L(\Q') \]
induced by inclusions into $\ZZ[x_1^{\pm1},...,x_n^{\pm1}]$.  This inclusion may be lifted to a ring homomorphism
\[ \mu: \ZZ[x_1,...,x_n,y_1,...,y_n]\rightarrow \ZZ[x_1,...,x_n,y_1,...,y_n]\]
\[ \mu(x_i)=x_i,\;\;\; \mu(y_i) = \left\{\begin{array}{cc}
y_i(p_i^++p_i^-) & \text{if }i\in Y \\
y_i & \text{otherwise}
\end{array}\right\}\]
with the property that $\pi'\circ \mu = \pi$, where $\pi'$ is the map
\[ \ZZ[x_1,x_2,...,x_n,y_1,y_2,...y_n]\longrightarrow \ZZ[x_1^{\pm1},x_2^{\pm1},...,x_n^{\pm1} ]\]
defined by $\Q'$ instead of $\Q$.  


Each of the variables appearing in the initial term of $f$ are fixed by $\mu$.  In lower-order terms of $f$, $\mu$ may introduce monomials in $x$; however, this will never create a term greater than $\In_<(f)$.  Hence, 
\[ \In_<(f) = \In_<(\mu(f)) \in \In_<(K_{\Q'})\]

Since $\Q'$ is acyclic, it was shown in \cite[Corollary 1.17]{BFZ05} that $\In_<(K_{\Q'})$ is generated by $\{x_iy_i \mid i\in [n]\}$.  Hence, $\In_<(f)$ is a multiple of $x_iy_i$ for some $i$.  However, this implies that $\In_<(f)$ is a multiple of the initial term of the $i$th defining polynomial in $K_\Q$, contradicting the assumption that $\In_<(f)\not\in M$.

It follows that $\In_<(K_\Q)\subset M$.  This consequently implies that $\In_<(J)=\In_<(K_\Q)$ and, by the preceding lemma, that $J=K_\Q$.  Furthermore, since $\In_<(G)$ generates $\In_<(K_\Q)$, the set $G$ is a Gr\"obner basis for $K_\Q$.
\end{proof}

\section{Simplicial Complexes and Cohen-Macaulayness}\label{sect:combinatorics}

\noindent
Now that we have obtained a generating set for $\In_< K_{\Q}$, we can explicitly construct the Stanley-Reisner complex of $\In_< K_{\Q}$. We first consider a larger class of simplicial complexes, defined as follows:
\begin{defn}
Let $S = \{1,\dots,n\}$, let $\sC$ be a collection $\{C_1,\dots,C_k\}$ of subsets of $S$, and let $Y \subseteq S$.  Define the simplicial complex $\Delta(S,\sC,Y)$ on the set $\{x_i \st i \in S\} \cup \{y_i \st i \in Y\}$ by the rule\footnote{If a subset $C_j$ is not contained in $Y$, we may simply ignore the condition $\{y_i\}_{i\in C_j} \not\subseteq S$, which is vacuously true because $y_i$ is not defined for $i\not\in Y$.}
\[
F \in \Delta(S,\sC,Y) \iff \forall i \, \{x_i,y_i\} \not \subseteq F \text{ and } \forall j \, \{y_i\}_{i \in C_j} \not \subseteq F.
\]
\end{defn}
Since every facet of $\Delta(S,\sC,Y)$ is of the form $\{z_1,\dots,z_n\}$, where $z_i$ is either $x_i$ or $y_i$, we see that $\Delta(S,\sC,Y)$ is always a pure simplicial complex. Note that for any quiver $\Q$ on vertex set $S$, where $\sC$ is the collection of sets of vertices of vertex-minimal directed cycles on $\Q$, we have by Theorem \ref{basisthm} that
\begin{equation}\label{idealform}
\In_< K_{\mathsf Q} = \left \la x_1y_1,\dots,x_ny_n, \prod_{i \in C_1}y_i,\dots,\prod_{i \in C_k}y_i \right \ra, \quad C_j \subseteq S,
\end{equation}
and the Stanley-Reisner complex of $\In_< K_{\mathsf Q}$ is precisely $\Delta(S,\sC,S)$.  

\begin{rem}
Whenever $\{i\}\in \sC$, there is no vertex of the form $y_i$ in the simplicial complex $\Delta(S,\sC,Y)$.  Such confusing notation is necessary for later induction.  A vertex in $\Delta(S,\sC,Y)$ of the form $y_i$ will be called a \textbf{$y$-vertex}.
\end{rem}

We now recall some definitions.
Given a simplicial complex $\Delta$ and a vertex $v$ of $\Delta$, the \textbf{link} of $v$ is the set
\[\link_\Delta(v) \coloneqq \{F \in \Delta \st F \not\ni v \text{ and } F \cup \{v\} \in \Delta\},
\]
and the \textbf{deletion} of $v$ is the set
\[
\del_{\Delta}(v) \coloneqq \overline{\{F \in \Delta \st F \cup \{v\} \notin \Delta\}},
\]
where the bar denotes closure, so that $\del_{\Delta}(v)$ is a simplicial complex. 
We call a vertex $v$ of a simplicial complex $\Delta$ a \textbf{shedding vertex} of $\Delta$ if no face of $\link_\Delta(v)$ is a facet of $\del_\Delta(v)$. Finally, we recall the (recursive) notion of vertex-decomposability: a simplicial complex $\Delta$ is \textbf{vertex-decomposable} if it is a simplex, or if it has a shedding vertex $v$ such that both $\link_\Delta(v)$ and $\del_\Delta(v)$ are vertex-decomposable (see \cite{BilleraProvan}, also \cite{BjornerWachs}). 

It is our goal to prove the following theorem, from which Theorem \ref{thm:CM} will follow.

\begin{thm}\label{homeo}
The complex $\Delta(S,\sC,Y)$ is always homeomorphic to a vertex-decomposable $(n-1)$-ball, except when $\sC = \emptyset$ and $Y=S$, in which case $\Delta(S,\sC,Y)$ is homeomorphic to a vertex-decomposable $(n-1)$-sphere.
\end{thm}

Note that the case where $\sC = \emptyset$ and $Y=S$ is precisely the case in which $\{y_j\}_{j \in S}$ is a face of $\Delta(S,\sC,Y)$. We first characterize the link and the deletion in $\Delta(S,\sC,Y)$ for any vertex of the form $y_i$ for $i\in Y$.

\begin{prop}\label{link}
For $y$-vertex $y_i$ in $\Delta(S,\sC,Y)$, we have $\link_{\Delta(S,\sC,Y)}(y_i) = \Delta(S^i,\sC^i,Y^i)$, where $S^i \coloneqq S \smallsetminus \{i\}$, $\sC^i \coloneqq \{C_j \cap S_i \st C_j \in \sC\}$, and $Y^i \coloneqq Y \cap S^i$.
\end{prop}
\begin{proof}
We first show $\Link \coloneqq \link_{\Delta(S,\sC,Y)}(y_i) \subseteq \Delta(S^i,\sC^i,Y^i)$. Since $\Link$ is a subcomplex of $\Delta(S,\sC,Y)$, no face of $\Link$ contains $\{x_j,y_j\}$ for any $j$, since $\Delta(S,\sC,Y)$ is defined so as never to contain any such face. Since no face of $\Link$ may contain either $x_i$ or $y_i$, we see that $\Link$ is a simplicial complex on $\{x_j \st j \in S^i\} \cup \{y_j \st j \in Y^i\}$. Finally, were some $F \in \Link$ to contain  $\{y_j\}_{j \in C_\ell^i}$ for some $\ell$, then $F \cup \{y_i\}$ would contain $\{y_j\}_{j \in C_\ell}$, contradicting $F \cup \{y_i\} \in \Delta(S,\sC,Y)$. We now have that $\Link \subseteq \Delta(S^i,\sC^i,Y^i)$.\\
\indent We now show $\Delta(S^i,\sC^i,Y^i) \subseteq \Link$. Consider some $F \in \Delta(S^i,\sC^i,Y^i)$. Clearly $F \not\ni y_i$, so suppose that $F \cup \{y_i\} \notin \Delta(S,\sC,Y)$. Then either $\{x_i,y_i\} \subseteq F \cup \{y_i\}$ or $\{y_j\}_{j \in C_\ell} \subseteq F \cup \{y_i\}$ for some $\ell$. The former case is impossible since $x_i \notin S^i$. The latter case implies $\{y_j\}_{j \in C_\ell^i} \subseteq F$, contradicting the definition of $\Delta(S^i,\sC^i,Y^i)$. Therefore we must have $F \cup \{y_i\} \in \Delta(S,\sC,Y)$ for every face $F$ of $\Delta(S^i,\sC^i,Y^i)$, from which it follows that $\Delta(S^i,\sC^i,Y^i) \subseteq \Link$. We conclude that $\Delta(S^i,\sC^i,Y^i) = \Link$.
\end{proof}

\begin{prop}\label{del}
For $y$-vertex $y_i$ in $\Delta(S,\sC,Y)$, we have $\del_{\Delta(S,\sC,Y)}(y_i) = \Delta(S,\sC,Y^i)$, where $Y^i$ is defined as above.
\end{prop}
\begin{proof}
We first show $\D \coloneqq \del_{\Delta(S,\sC,Y)}(y_i) \subseteq \Delta(S,\sC,Y^i)$. Since no face of $\D$ may contain $y_i$, we see that $\D$ is a simplicial complex on $\{x_j \st j \in S\} \cup \{y_j \st j \in Y^i\}$. Since $\D$ is a subcomplex of $\Delta(S,\sC,Y)$, no face of $\D$ contains either $\{x_j,y_j\}$ for any $j$ or $\{y_j\}_{j \in C_\ell}$ for any $\ell$. Therefore $\D \subseteq \Delta(S,\sC,Y^i)$.\\
\indent Since $\Delta(S,\sC,Y^i)$ has $x_i$ as a vertex but not $y_i$, we have by the definition of $\Delta(S,\sC,Y^i)$ that every facet of $\Delta(S,\sC,Y^i)$ contains $x_i$. Consider some arbitrary facet $F$ of $\Delta(S,\sC,Y^i)$. Since $x_i \in F$, we cannot have $F \cup \{y_i\} \in \Delta(S,\sC,Y)$, and so $F \in \D$. Therefore $\D \subseteq \Delta(S,\sC,Y^i)$, and so we conclude that $\D = \Delta(S,\sC,Y^i)$.
\end{proof}
We may now observe an important relationship between links and deletions that arises in our case. The following result shows that any vertex $y_i$ is a shedding vertex. Note that in the case where $\sC = \emptyset$ and $Y=S$, every vertex $y_i$ must be a shedding vertex because its link is always empty.

\begin{lem}\label{boundary}
Except in the case where $\sC = \emptyset$ and $Y=S$, the complex $\Delta(S^i,\sC^i,Y^i)$ is properly contained in the boundary complex $\dee \Delta(S,\sC,Y^i)$.
\end{lem}
\begin{proof}
Since $\Delta(S^i,\sC^i,Y^i) \subseteq \Delta(S,\sC,Y^i)$ and $\Delta(S,\sC,Y^i)$ is pure, it follows that every facet of $\Delta(S^i,\sC^i,Y^i)$ meets at least one facet of $\Delta(S,\sC,Y^i)$. Now we show that each facet of $\Delta(S^i,\sC^i,Y^i)$ meets only one facet of $\Delta(S,\sC,Y^i)$. Observe that the facets of $\dee \Delta(S,\sC,Y^i)$ are characterized as the codimension 1 faces $\{z_j\}_{j \ne k}$, $k \in \{1,\dots,n\}$, where each $z_j$ is either $x_j$ or $y_j$, such that exactly one of either $\{z_j\}_{j \ne k} \cup \{x_k\}$ or $\{z_j\}_{j \ne k} \cup \{y_k\}$ lies in $\Delta(S,\sC,Y^i)$. Every facet of $\Delta(S^i,\sC^i,Y^i)$ is of the form $\{z_j\}_{j \ne i}$, and it always happens that $\{z_j\}_{j \ne i} \cup \{x_i\} \in \Delta(S,\sC,Y^i)$ and $\{z_j\}_{j \ne i} \cup \{y_i\} \notin \Delta(S,\sC,Y^i)$. Therefore $\Delta(S^i,\sC^i,Y^i) \subseteq \dee \Delta(S,\sC,Y^i)$.\\
\indent We now must show that this containment is proper. We have two cases. Either $\{y_j\}_{j \ne i}$ is a face of $\Delta(S,\sC,Y^i)$ or it is not. If it is, then it must lie on $\dee \Delta(S,\sC,Y^i)$, because $\{y_j\}_{j \ne i} \cup \{x_i\}$ is a face of $\Delta(S,\sC,Y^i)$, while $\{y_j\}_{j \ne i} \cup \{y_i\}$ is not a face of $\Delta(S,\sC,Y^i)$ since either $\sC \ne \emptyset$ or $Y \ne S$. If $\{y_j\}_{j \ne i}$ is not a face of $\Delta(S,\sC,Y^i)$, then there must be some $C \in \sC$ not containing $i$ such that no other member of $\sC$ is a subset of $C$. Then, for any $k \in C$, we have that $F = \{x_j\}_{j \notin C} \cup \{y_j\}_{j \in C\smallsetminus\{k\}}$ is a face of $\Delta(S,\sC,Y^i)$. This face $F$ lies on $\dee \Delta(S,\sC,Y^i)$, since $F \cup \{x_k\} \in \Delta(S,\sC,Y^i)$ but $F \cup \{y_k\} \notin \Delta(S,\sC,Y^i)$. Since both $\{y_j\}_{j \ne i}$ and $F$ contain $x_i$, neither is a face of $\Delta(S^i,\sC^i,Y^i)$. Therefore there is always an element of $\dee \Delta(S,\sC,Y^i)$ that is not in $\Delta(S^i,\sC^i,Y^i)$, and so the containment $\Delta(S^i,\sC^i,Y^i) \subseteq \dee \Delta(S,\sC,Y^i)$ is proper.
\end{proof}
By the previous lemma and the remarks above, we see that every $\Delta(S,\sC,Y)$ is vertex-decomposable, because the $y$-vertices are always shedding vertices, and any complex without $y$-vertices is a simplex.  The remainder of the proof is to strengthen this argument to prove that these simplicial complexes are balls or spheres, as appropriate.


\begin{proof}[Proof of Theorem \ref{homeo}]
First, we prove that $\Delta(S,\sC,Y)$ is a vertex-decomposable $(n-1)$-ball when $Y\neq S$ or $\sC\neq \emptyset$, by induction on the number of $y$-vertices.

If there are no $y$-vertices in $\Delta(S,\sC,Y)$ (that is, $\{i\}\in \sC$ for all $i\in Y$), then $\Delta(S,\sC,Y)$ is just one simplex on $n$ vertices, which is homeomorphic to an $(n-1)$-ball. Assume the inductive hypothesis holds whenever there are fewer than $m$ $y$-vertices, and assume that $\Delta(S,\sC,Y)$ has $m$-many $y$-vertices.  Choose a vertex $y_i$ in $\Delta(S,\sC,Y)$, and define 
\[\Link \coloneqq \link_{\Delta(S,\sC,Y)}(y_i) = \Delta(S^i,\sC^i,Y^i)\text{ and } \D \coloneqq \del_{\Delta(S,\sC,Y)}(y_i) = \Delta(S,\sC,Y^i)\]
We observe that both $\Link$ and $\D$ satisfy the inductive hypothesis; this is clear when $Y\neq S$.  If $Y=S$, then by assumption $\sC\neq\emptyset$.  Since $y_i$ is a vertex of $\Delta(S,\sC,Y)$, we also know that $\{i\}\not\in \sC$.  It follows that $\sC^i\neq\emptyset$, and so $\Link$ still satisfies the inductive hypothesis.  Therefore, $\Link$ is a vertex-decomposable $(n-2)$-ball and $\D$ is a vertex-decomposable $(n-1)$-ball.


%
%

As a consequence, the cone $\Cone$ from $y_i$ on $\link_{\Delta(S,\sC,Y)}(y_i)$ is a vertex-decomposable $(n-1)$-ball. By Lemma \ref{boundary}, $\Cone$ and $\D$ meet at the proper subset $\Link$ of $\dee \D$, which is a vertex-decomposable $(n-2)$-ball. Therefore, $\Delta(S,\sC,Y) = \Cone \cup \D$ is a vertex-decomposable $(n-1)$-ball, completing the induction.

The remaining case is $\Delta(S,\emptyset,S)$. Consider the mapping $\{x_1,\dots,x_n,y_1,\dots,y_n\} \to \R^n$ given by $x_i \mapsto e_i$ and $y_i \mapsto -e_i$, where $\{e_1,\dots,e_n\}$ is the standard basis for $\R^n$. This mapping induces a bijection between the faces of $\Delta(S,\emptyset,S)$ and the faces of the cross-polytope (i.e. orthoplex) on the vertices $\{e_1,\dots,e_n,-e_1,\dots,-e_n\}$. Since this figure is homeomorphic to an $(n-1)$-sphere, so must be $\Delta(S,\emptyset,S)$.
\end{proof}
As noted, by Theorem \ref{basisthm} we have that the initial ideal of any lower bound ideal is of the form \eqref{idealform}. Therefore, by Theorem \ref{homeo}, the Stanley-Reisner complex of the initial ideal of any lower bound ideal is homeomorphic to either a ball or a sphere. It follows that all lower bound algebras over a field are Cohen-Macaulay, and so Theorem \ref{thm:CM} holds. 

\section{Normality of lower bound algebras}\label{sect:normality}

In this section, we prove that all lower bound algebras defined from a quiver are normal. As explained in the introduction, the case where $\Q$ is acyclic follows immediately because $\mathcal{L}(\Q)$ is equal to its upper cluster algebra, and is therefore normal. So, for the remainder of the section, we assume that $\Q$ contains a cycle.
In this case, our proof of normality relies on a very slight adaptation of \cite[Proposition 8.1]{KLS-Richardson}. 



\begin{prop}\label{prop:normalityInGeneral}(cf. \cite[Proposition 8.1]{KLS-Richardson})
Fix a monomial order $<$ on the polynomial ring $\mathbb{K}[z_1,\dots,z_n]$. 
Let $X$, and $Y_1,\dots, Y_r$ be (reduced and irreducible) closed affine subvarieties of $\mathbb{A}^n$, where each of $Y_1,\dots, Y_r$ are codimension-$1$ in $X$.
Assume that, with respect to the term order $<$, each of $X$ and $Y_1,\dots, Y_r$ Gr\"obner degenerate to Stanley-Reisner schemes. Then, if
\begin{enumerate}
\item the Stanley-Reisner complex of $\In_<X$ is a simplicial ball; 
\item the Stanley-Reisner complex of each $\In_<Y_i$ lies entirely on the boundary sphere $\partial \Delta_X$; and
\item $X\smallsetminus (Y_1\cup Y_2\cup \cdots \cup Y_r)$ is normal, 
\end{enumerate}
then $X$ is normal.
\end{prop}

We need the following standard result to prove this proposition.  It is very similar to \cite[Proposition 3.1 (b)]{BC03}; we provide the necessary modifications in the proof below.

\begin{lem}
Fix a monomial order $<$ on the polynomial ring $S \coloneqq \mathbb{K}[z_1,\dots,z_n]$. 
Let $X$ and $Y$ be irreducible affine subvarieties of $\mathbb{A}^n$, and assume that $Y$ is codimension-$1$ in $X$. 
Then if $\textrm{in}_< X$ is generically regular along each irreducible component of $\textrm{in}_< Y$ then $X$ is generically regular along $Y$.
\end{lem}

\begin{proof}
Let $X = \mathbb{V}(I)$ and $Y = \mathbb{V}(J)$ for $I,J\subseteq \mathbb{K}[z_1,\dots, z_n]$.
Pick a weight vector $\lambda$ such that $\textrm{in}_\lambda I = \textrm{in}_< I$ and $\textrm{in}_\lambda J = \textrm{in}_< J$.
Let $f = \sum_i a_im_i$, where each $a_i\in \mathbb{K}$, and each $m_i$ is a monomial. Let $\textrm{hom}_{\lambda}(f)$ denote the $\lambda$-homogenization of $f$ inside $S[t]$, that is,
\[\textrm{hom}_\lambda(f) \coloneqq \sum a_im_it^{\lambda(f)-\lambda(m_i)},\] where $\lambda(f)$ denotes the highest $\lambda$-weight of any monomial in $f$, and $\lambda(m_i)$ is the $\lambda$-weight of the monomial $m_i$.
Let $\textrm{hom}_\lambda I$ denote the $\lambda$-homogenization of the ideal $I$, that is,
$\textrm{hom}_\lambda I \coloneqq \langle \textrm{hom}_\lambda(f) \mid f\in I \rangle.$ 
It is a standard fact that $A \coloneqq S[t]/\textrm{hom}_{\lambda} I$ is a free $\mathbb{K}[t]$-module and that $A/\langle t\rangle \cong S/\textrm{in}_{\lambda} I$ (see eg. \cite[Proposition 2.4]{BC03} or \cite[Theorem 15.17]{Eis95}).

Now, by assumption, $\textrm{in}_< X$ is generically regular along each irreducible component of $\textrm{in}_< Y$. That is, the localization of $S/\textrm{in}_{<} I$ at any minimal prime of $\textrm{in}_< J$ is a regular local ring. 
Thus, by the above facts, we have that the localization of $A/\langle t\rangle$ at any minimal prime of $(\textrm{hom}_\lambda J+\langle t\rangle)$ is a regular local ring. 

Observe that $A$ is positively graded. Let $\frak{m}$ denote the maximal ideal generated by the indeterminates $z_1,...,z_n,t$, and let $A_{\frak{m}}$ denote the localization at $\frak{m}$. 
Because $A/\langle t \rangle$ localized at any minimal prime $\frak{p}$ of $(\textrm{hom}_\lambda J+\langle t\rangle)$ is regular, so too is $A_\frak{m}/\langle t\rangle$ localized at any non-trivial $A_\frak{m} \frak{p}$, and the non-trivial $A_\frak{m}\frak{p}$ are precisely the minimal primes of $(\textrm{hom}_\lambda J+\langle t\rangle)$ as an ideal in $A_\frak{m}/\langle t\rangle$.

Now we can use the proof of \cite[Lemma 3.2]{BC03} to get that the localization of $A_{\frak{m}}$ at the height-$1$ prime ideal $\textrm{hom}_\lambda J$ is regular. The second half of the proof of \cite[Proposition 3.1 (b)]{BC03} then gives that the localization of $S/I$ at $J$ is regular.
\end{proof}

We now prove Proposition \ref{prop:normalityInGeneral}.

\begin{proof}[Proof of Proposition \ref{prop:normalityInGeneral}]
We follow the proof given in \cite[Proposition 8.1]{KLS-Richardson}. 
To show that $X$ is normal, we need to show that $X$ is $R1$ and $S2$. Since $\Delta_X$ is a simplicial ball by assumption (i), it follows that $X$ is Cohen-Macaulay, and hence $S2$.
To show that $X$ is $R1$, first observe that, by assumption (iii), if $\frak{p}\subseteq S/I$ is a prime ideal of height $\leq 1$ which \emph{is not} the generic point of any $Y_i$, then $(S/I)_{\frak{p}}$ is regular. 

The remaining primes in $S/I$ which have height $\leq 1$ are the generic points of the various $Y_i\subseteq X$.
It therefore remains to show that $X$ is generically regular along each irreducible subvariety $Y_i$. 
By assumption (ii), we have that $\textrm{in}_< X$ is generically regular along each irreducible component of each $\textrm{in}_< Y_i$. Thus, by the lemma, we get that $X$ is generically regular along $Y_i$.
\end{proof}

To use Proposition \ref{prop:normalityInGeneral} to prove that all lower bound algebras are normal, we need to show that the hypotheses (ii) and (iii) always hold for quivers with cycles. We start with (ii). To show the desired result, we use results of Knutson from \cite{Knu09}\footnote{The statement given here is less general than the one that appears in \cite[Theorem 2]{Knu09} and \cite[Lemma 6]{Knu09}. We also change the hypotheses of \cite[Theorem 2]{Knu09}, however, this is harmless as the proof goes through in the exact same way.}. 

\begin{thm}\label{thm:Knutson}(cf. Theorem 2, Lemma 6, Corollary 2 of \cite{Knu09})
Let $f \in \mathbb{Z}[z_{1}, \ldots, z_{n}]$ be a polynomial with the property that, for each prime $p$, $f^{p-1} (\text{mod }p)$ has a unique term divisible by $z_1^{p-1}z_2^{p-1}\cdots z_n^{p-1}$, and let $<$ be a term order of $\mathbb{Z}[z_1,\dots,z_n]$ for which $\textrm{in}_<f = z_1z_2\cdots z_n$.
Denote by $\mathcal{J}$ the smallest set of ideals that contains the ideal $\langle f \rangle$
and such that 
\begin{enumerate}
\item  if $I_{1}, I_{2} \in \mathcal{J}$, then $I_{1} + I_{2}, I_{1} \cap I_{2} \in \mathcal{J}$; and
\item  if $I \in \mathcal{J}$ and $J$ is a primary component of $I$ then
$J \in \mathcal{J}$.
\end{enumerate}
Then, over any field $\mathbb{K}$, every ideal $J \in \mathcal{J}$ is a radical ideal and
the initial ideal of every $J \in \mathcal{J}$ with respect to $<$ is a squarefree monomial ideal.
Furthermore, for any $I_1$ and $I_2$ in  $\mathcal{J}$,
\[\textrm{in}_<(I_1\cap I_2) = \textrm{in}_< I_1 \cap \textrm{in}_< I_2,\textrm{ and }\textrm{in}_<(I_1+I_2) = \textrm{in}_< I_1+\textrm{in}_< I_2.\]
\end{thm}


We will make use of this theorem in the case where $f = \prod_{i=1}^n (x_iy_i-p_i^+-p_i^-)$. Here we take $<$ to be a weighting of the variables where the $y$-variables are much more expensive than the $x$-variables. Observe that $f$ and $<$ satisfy the assumptions of Theorem \ref{thm:Knutson} and the ideal
\[ 
I_\Q := \langle x_iy_i-p_i^+-p_i^-\mid 1\leq i\leq n \rangle.
\]
lies in the collection of ideals $\mathcal{J}$ from the theorem. 
Consequently $I$ is radical. We may then write an irredundant prime decomposition 
\begin{equation}\label{primeDecompositionl}
I_\Q = K_\Q\cap P_1\cap\cdots \cap P_r
\end{equation}
where each $P_i$ is a minimal prime, and $K_\Q$ is the lower bound ideal \cite[Lemma 5.7]{BMRS15}. Consequently, each $P_i+K_\Q\in \mathcal{J}$ and so each $P_i+K_\Q$ is radical and degenerates to a squarefree monomial ideal. 

\begin{prop}\label{prop:boundarySphere}
Let $\Q$ be a quiver with a directed cycle, so that there is at least one prime $P_i$ in \eqref{primeDecompositionl}.
With respect to a term order where the $y$-variables are much more expensive than the $x$-variables, each prime component of $P_i+K_\Q$ Gr\"obner degenerates to the Stanley-Reisner ideal of a sub-simplicial complex of $\partial \Delta_{K_\Q}$. 
Furthermore,
$\textrm{in}_< ((P_1\cap\cdots \cap P_r) + K_\Q)$ is the Stanley-Reisner ideal of the entire boundary $\partial \Delta_{K_\Q}$.
\end{prop}

\begin{proof}
By Theorem \ref{thm:Knutson}, we have that $P_i+K_\Q$ is radical and Gr\"obner degenerates to a squarefree monomial ideal. Let 
\[
K_\Q+P_i = \cap_J J
\]
be a decomposition of $K_\Q+P_i$ into minimal primes. By Theorem \ref{thm:Knutson}, each $\textrm{in}_< J$ is a squarefree monomial ideal. Applying the second part of Theorem \ref{thm:Knutson} and translating into the language of simplicial complexes yields the equality
\[
\Delta_{K_\Q+P_i} = \Delta_{K_\Q}\cap \Delta_{P_i} = \cup_J \Delta_J
\]
where  $\Delta_J$ denotes the Stanley-Reisner complex $\textrm{in}_< J$. To prove the first claim of the proposition, we must show that every face of each $\Delta_J$ is contained in the boundary sphere of the simplicial ball $\Delta_{K_\Q}$. 
So, suppose otherwise, and let $F$ be a face of some $\Delta_J$ which is not contained in the boundary $\partial \Delta_{K_\Q}$. Assume that $F$ is a maximal such face. We claim that $F$ must be a facet of $\Delta_{P_i}$.  

To prove this claim, we first apply Theorem \ref{thm:Knutson} to the prime decomposition in equation \eqref{primeDecompositionl} to get
\begin{equation}\label{eq:simplicialEquality}
\textrm{in}_< I = \textrm{in}_< (K_\Q)\cap \textrm{in}_<(P_1)\cap\cdots \cap \textrm{in}_<(P_r)
\end{equation}
which, after translating into the language of simplicial complexes, says that $\Delta_{K_\Q}$ and every $\Delta_{P_i}$ is contained in the Stanley-Reisner complex associated to $\textrm{in}_< I = \langle x_iy_i \mid 1\leq i\leq n \rangle$, which can be geometrically realized as the $(n-1)$-dimensional boundary sphere of a cross-polytope on $2n$ vertices. Decompose this simplicial sphere into the union of two $(n-1)$-dimensional simplicial balls:
\[
\Delta_I = \Delta_{K_\Q}\cup C, \textrm{  where  } C \coloneqq \overline{\Delta_I\smallsetminus \Delta_{K_\Q}}.
\]
Observe that, by construction, $\Delta_{K_\Q}\cap C$ is the boundary sphere of $\Delta_{K_\Q}$.


Now, suppose that $F$ is not a facet of $\Delta_{P_i}$. Then there is a vertex $z$ such that $F\cup\{z\}$ is a face of $\Delta_{P_i}$.
Then, using the decomposition of $\Delta_I$, we see that either $F\cup\{z\}$ is contained in $\Delta_{K_\Q}$, or it is contained in $C$. If $F\cup \{z\}\subseteq \Delta_{K_\Q}$, we contradict the maximality of $F$. If $F\cup\{z\}\in C$, we contradict that $F$ was not contained in the boundary of $\Delta_{K_\Q}$ (since $\Delta_{K_\Q}$ and $C$ only intersect along the boundary of $\Delta_{K_\Q}$).

Thus, our maximal face $F$ must be a facet of $\Delta_{P_i}$, which, since $P_i$ is prime, must have dimension one less than the dimension $\mathrm{dim}(S/P_i)$. But this is not possible because $\mathrm{dim}(S/J)$ is strictly smaller than $\mathrm{dim}(S/P_i)$.

To obtain the last statement, we translate equality (\ref{eq:simplicialEquality}) into the language of simplicial complexes to see that the union $\cup_{i=1}^r \Delta_{P_i}$ necessarily contains the boundary sphere $\partial \Delta_{K_\Q}$. Thus, so does 
\[\Delta_{(P_1\cap\cdots \cap P_r)+K_\Q)} = \Delta_{(P_1\cap\cdots\cap P_r)}\cap \Delta_{K_\Q} = \cup_{i=1}^r (\Delta_{P_i}\cap \Delta_{K_\Q}). \]
But, as already shown, each $\Delta_{P_i}\cap \Delta_{K_\Q}$ is contained inside of the boundary sphere of $\Delta_{K_\Q}$ and so we are done.
\end{proof}

We next show that (iii) of Proposition \ref{prop:normalityInGeneral} holds for lower bound algebras.

\begin{prop}\label{prop:openSetIsNormal}
Let $\mathbb{V}(K_\Q)$ denote the lower bound variety of a quiver $\Q$.  Then $\mathbb{V}(K_\Q)\smallsetminus \mathbb{V}(P_1\cap P_2 \cap \cdots \cap P_r)$ is normal\footnote{We note that $\mathbb{V}(P_1\cap P_2\cap\cdots\cap P_r)$ is empty when $\Q$ is acyclic.}.
\end{prop}
\begin{proof}
Consider $q\in \mathbb{V}(K_\Q)$, and define the (possibly empty) set 
\[ S_q \coloneqq \{ i\in \{1,2,...,n \} \mid x_i(q) =0 \} \]
The vertices indexed by $S_q$ must be unfrozen, since frozen $x$-variables are invertible.

First, assume $S_q$ does not contain a directed cycle, and consider the open set
\[ U_q \coloneqq \{ q'\in \mathbb{V}(K_\Q) \mid \forall i\not\in S_q,\;x_i(q)\neq 0\} \]
The coordinate ring of $U_q$ is the localization of $\L(\Q)$ at the set of $x$-variables which are not in $S_q$; hence, it is isomorphic to the lower bound algebra of the ice quiver $\Q^\dagger$ obtained by freezing the vertices not in $S_q$.  This ice quiver is \emph{acyclic}, and so the lower bound algebra coincides with the upper cluster algebra \cite{BFZ05}, which is normal \cite{MulLA}.  Hence, $\mathbb{V}(K_\Q)$ is normal at $q$.

Next, assume $S_q$ contains a directed cycle, and consider the affine space 
\[ \mathbb{W}_q \coloneqq \{ q' \in \mathbb{K}^{2n} \mid \forall i, \;x_i(q')=x_i(q),\text{ and } \forall i \not\in S_q,\; y_i(q')=y_i(q)\} \]
This contains $q$ and is contained in $\mathbb{V}(K_\Q\cap P_1 \cap \cdots \cap P_r)$.  Since $S_q$ contains a directed cycle, there is some cycle polynomial whose leading term is a product of $y$-variables whose indices are contained in $S_q$, and hence cannot vanish on $\mathbb{W}_q$.  Hence, $\mathbb{W}_q\not\subset \mathbb{}V(K_\Q)$. By the irreducibility of $\mathbb{W}_q$, we have that $q\in\mathbb{W}_q\subset \mathbb{V}(P_1\cap \cdots\cap P_r)$, and so $q\not\in \mathbb{V}(K_\Q)\smallsetminus \mathbb{V}(P_1\cap \cdots \cap P_r)$.
\end{proof}

We can now prove that all lower bound algebras are normal.

\begin{proof}[Proof of Theorem \ref{thm:normalityOfLowerBounds}]
We have already treated the case when $\Q$ is acyclic (i.e. $\mathcal{L}(\Q)$ equals its associated upper cluster algebra, and is hence normal). So assume that $\Q$ contains a cycle.
Let $K_\Q$ be the relevant lower bound ideal, and let $P_1,\dots, P_r\subseteq S$ be minimal primes of $\langle x_iy_i-p_i^+-p_i^-\mid 1\leq i\leq n \rangle$ such that 
\[
\langle x_iy_i-p_i^+-p_i^-\mid 1\leq i\leq n \rangle = K_\Q\cap P_1\cap\cdots \cap P_r.
\]
Now apply Proposition \ref{prop:normalityInGeneral} with $X = \mathbb{V}(K_\Q)$, and $Y_1,\dots,Y_s$ the irreducible components of the various $\mathbb{V}(P_j + K_\Q)$, $1\leq j\leq r$. Observe that item (i) of Proposition \ref{prop:normalityInGeneral} follows from Theorem \ref{thm:CM}, item (ii) follows from Proposition \ref{prop:boundarySphere}, and item (iii) holds by Proposition \ref{prop:openSetIsNormal}.
\end{proof}

\appendix

\section{The singular locus of $\L(\Q)$}\label{section: singularity}

The Cohen-Macaulayness and normality of $\L(\Q)$ may both be regarded as constraints on the singularities of the variety $\mathbb{V}(K_\Q)$.  It is then natural to directly consider the singularities of $\mathbb{V}(K_\Q)$.
This appendix provides an example to demonstrate that, even in the most elementary cases, even the existence of singularities in $\mathbb{V}(K_\Q)$ can be difficult to predict, suggesting a direct study of of the singularities of $\mathbb{V}(K_\Q)$ may be daunting.

Fix an algebraically closed field $\mathbb{K}$.
Let $\Q_n$ denote the ice quiver with vertex set $\{1,2,...,n\}$, an arrow from $i$ to $i+1$ for each $1\leq i <n$, and no frozen vertices (see Figure \ref{fig: Atype}).  

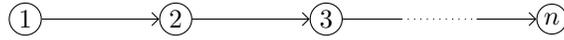
\begin{figure}[h!t]
\begin{tikzpicture}
	\node[mutable] (1) at (0,0) {$1$};
	\node[mutable] (2) at (2,0) {$2$};
	\node[mutable] (3) at (4,0) {$3$};
	\node[mutable] (5) at (7,0) {$n$};
	\draw[-angle 90] (1) to (2);
	\draw[-angle 90] (2) to (3);
	\draw[-] (3) to (5,0);
	\draw[-angle 90] (6,0) to (5);
	\draw[dotted] (5,0) to (6,0);
\end{tikzpicture}
\caption{The quiver $\Q_n$}
\label{fig: Atype}
\end{figure}

\begin{prop}
The $\mathbb{K}$-variety $\mathbb{V}(K_{\Q_n})$ has a unique singular point when $n\equiv 3\; (\text{mod }4)$, and no singularities otherwise.
\end{prop}

\begin{proof}
Since there are no directed cycles, we have the following presentation of $\L(\Q_n)$.
\[ \L(\Q_n) = \mathbb{K}[x_1,...,x_n,y_1,..,y_n]/\langle y_1x_1-x_2-1,y_2x_2-x_3-x_1,...,y_{n-1}x_{n-1}-x_n-x_{n-2},y_nx_n-1-x_{n-1}\rangle \]
Let $p\in \mathbb{V}(K_{\Q_n})$.  First, we observe that $x_i(p)$ and $x_{i+1}(p)$ cannot both be zero.  This is clear for $i=1$ from the defining relation for $y_1$, and the general case follows by induction on $i$.

The point $p$ is singular if and only if the associated 
\emph{Jacobian matrix} $Jac_p$ has rank less than $n$.
%
\[ Jac_p:=\gmat{ 
y_1(p) & -1 & 0 & 0 & \cdots  & x_1(p) & 0 & 0 & 0 & \cdots 
\\ -1 & y_2(p) & -1 & 0 & \cdots & 0 & x_2(p) & 0 & 0& \cdots
\\ 0 & -1 & y_3(p) & -1 & \cdots & 0 & 0 & x_3(p) & 0 & \cdots
\\ 0 & 0 & -1 & y_4(p) & \cdots & 0 & 0 & 0 & x_4(p)
\\ \vdots & \vdots & \vdots & \vdots & \ddots & \vdots & \vdots & \vdots & \vdots & \ddots 
}\]
Equivalently, $p$ is singular if and only if there is a non-trivial linear relation among the rows of $Jac_p$.  Clearly, such a relation can only include the $i$th row if $x_i(p)=0$; hence, such a relation can only involve non-adjacent rows of $Jac_p$.  This is only possible when $n$ is odd and when $x_i(p)=y_i(p)=0$ for every odd $i$; furthermore, the relation (up to scaling) must be that the alternating sum of the odd rows of $Jac_p$ is $0$.

Returning to the defining relations, the condition that $x_i(p)=0$ for all odd $i$ is only possible when $n$ is congruent to $3$ mod $4$, in which case $x_{2j}(p)=(-1)^j$.  Since these are non-zero, it follows that $y_{2i}(p)=\frac{x_{2i-1}(p)+x_{2i+1}(p)}{x_{2i(p)}} = 0$.  Consequently, when $n\equiv 3\;(\text{mod }4)$, the unique singular point is given by
\[\forall i\in \{1,2,...,n\},\;\;\; x_i(p) = \left\{ \begin{array}{cc}
0 & \text{if $i$ is odd} \\
(-1)^{\frac{i}{2}} & \text{if $i$ is even}
\end{array}\right\},\;\;\; y_i(p) =0 \]
If $n\not\equiv 3\;(\text{mod }4)$, there is no singular point.
\end{proof}

The family of algebras $\L(\Q_n)$ is one of the most fundamental and elementary in the theory of cluster algebras; in this case, the lower bound $\L(\Q_n)$ coincides with the cluster algebra of Dynkin-type $A_n$.  This and other simple examples suggest that that presence of singularities in $\mathbb{V}(K_{\Q})$ is difficult to predict, and is very sensitive to small changes in the quiver $\Q$.

\section{Skew-symmetrizable lower bound cluster algebras}\label{section: skew}

\def\B{\mathsf{B}}

The body of this paper considered \emph{lower bound cluster algebras defined by ice quivers}, also called \emph{skew-symmetric lower bound cluster algebras}.  However, the proofs and results can be extended to the larger generality of \emph{skew-symmetrizable lower bound cluster algebras}.\footnote{We remark that we still work in \emph{geometric type}, rather than more exotic semifield coordinates.}  This appendix outlines the necessary modifications.  
A skew-symmetrizable lower bound algebra is defined by a \textbf{exchange matrix} $\B$: an $n\times m$ integer matrix (for $n\geq m$) such that there is an $m\times m$ diagonal matrix $\mathsf{D}$ with the property that the top $m\times m$ minor of $\B\mathsf{D}$ is skew-symmetric.

To each index $i\in \{1,2,...,m\}$, we associate a pair of monomials $p_i^+,p_i^-\in \ZZ[x_1,x_2,...,x_n]$ as follows.
\begin{equation}
p_i^+ \coloneqq \prod_{j\in \{1,2,...,n\}}x_j^{\max(\B_{ji},0)},\hspace{1cm} p_i^- \coloneqq\prod_{j\in \{1,2,...,n\}}x_j^{\max(-\B_{ji},0)}
\end{equation}
Each index then determines a Laurent polynomial $x_i'$, called the \textbf{adjacent cluster variable at $i$}, which is defined by the following formula.\footnote{The same apology is due as before. When $i>m$, the element $x_i'$ is not technically a cluster variable, but treating it as one streamlines the process.}
\begin{equation}\label{eq: mutation}
x_i' \coloneqq \left\{\begin{array}{cc}
x_i^{-1}(p_i^++p_i^-) & \text{if $i\in \{1,2,..., m\}$} \\
x_i^{-1} & \text{if $i\in \{m+1,m+2,...,n\}$}
\end{array}\right\}
\end{equation}
The \textbf{lower bound algebra} $\L(\B)$ defined by $\B$ is the subalgebra of $\ZZ[x_1^{\pm1},...,x_n^{\pm1}]$ generated by the variables $x_1,x_2,...,x_n$ and the adjacent cluster variables $x_1',x_2',...,x_n'$.
\begin{ex}
Consider the following $3\times 2$ matrix $\B$, which is skew-symmetrizable with the given $\mathsf{D}$.
\[ \B := \gmat{0 & 3 \\ -2 & 0 \\ 1 & 2},\;\;\; \mathsf{D}:= \gmat{ 3 & 0 \\ 0 & 2} \]
The three adjacent cluster variables are below.
\[ x_1' = \frac{x_3 + x_2^2}{x_1} ,\;\;\; x_2' = \frac{x_1^3x_3^2+1}{x_2},\;\;\; x_3' = x_3^{-1} \]
\end{ex}


%

To translate prior results into the generality of exchange matrices, we associate an ice quiver $\Q(\B)$ to $\B$.  The vertex set of $\Q(\B)$ is the set $\{1,2,...,n\}$, the frozen vertex set is $\{m+1,m+2,...,n\}$, and there is an arrow from $i$ to $j$ if $\B_{ji}>0$ or $\B_{ij}<0$ (and no other arrows).  The lower bound algebra $\L(\B)$ is \emph{not} the lower bound algebra of the ice quiver $\Q(\B)$; the arrows in $\Q(\B)$ only keep track of the sign of entries in $\B$.

The results and proof in this paper hold verbatim for $\L(\B)$, using $\Q(\B)$ in place of $\Q$, for two key reasons.
\begin{itemize}
	\item Our arguments never use the specific exponent of the monomials $p_i^\pm$; rather, we only need to know \emph{if} a variable $x_j$ divides $p_i^\pm$, which corresponds to the existence of an arrow in $\Q(\B)$.
	\item When $\Q(\B)$ is acyclic, the presentation of $\L(\B)$ given in \cite[Corollary 1.17]{BFZ05} and used in the proof of Theorem \ref{thm: grobner} is still valid in the larger generality of skew-symmetrizable cluster algebras.
\end{itemize}

\begin{rem}
The body of the paper is not in this larger generality for reasons of clarity and exposition, rather than any mathematical limitations.
\end{rem}

%
%
%

\bibliography{UniversalBib}
\bibliographystyle{amsalpha}




\end{document}